\documentclass[english]{article}
\usepackage{babel}
\usepackage{amssymb}
\usepackage{amsmath}
\usepackage{amsthm}
\usepackage{pstricks}
\usepackage{graphicx}
\usepackage{pst-node}

\usepackage{times} 

\begin{document}
\def\R{\mathbb R}
\def\N{\mathbb N}
\def\Z{\mathbb Z}
\def\C{\mathbb C}
\def\Co{\mathcal{C}}
\def\F{\mathcal{F}}
\def\S{\mathbb {S}}
\def\1{\mathbf{1}}
\newtheorem{th-def}{Theorem-Definition}[section]
\newtheorem{theo}{Theorem}[section]
\newtheorem{lemm}[theo]{Lemma}
\newtheorem{prop}[theo]{Proposition}
\newtheorem{defi}[theo]{Definition}
\newtheorem{cor}[theo]{Corollary}
\newtheorem{exam}[theo]{Example}
\newtheorem{Rem}[theo]{Remark}
\newtheorem{Proof of}[theo]{Proof of main Theorem}
\def\dem{\noindent \textbf{Proof: }}
\def\rem{\indent \textsc{Remark. }}
\def\rems{\indent \textsc{Remarks. }}
\def\fin{  $\square$}
\def\ep{\varepsilon}
\def\uep{u^{\varepsilon}}
\def\Wep{W^{\varepsilon}}
\def\qep{a^{\varepsilon}}
\def\Sep{S^{\varepsilon}}
\def\Qep{Q^{\varepsilon}}
\def\Zep{Z^{\varepsilon}}
\def\vep{v^{\varepsilon}}
\def\wep{w^{\varepsilon}}
\def\fep{f^{\varepsilon}}
\def\gep{g^{\varepsilon}}
\def\hep{H^{\varepsilon}}
\def\aep{\alpha_{\varepsilon}}
\def\ph{\phi^{\ep}}
\def\X{\mathbb{X}}
\def\Y{\mathbb{Y}}
\def\H{\mathcal{H}}
\newcommand{\Node}[1]{\makebox[3mm]{#1}}
\def\cvput#1[#2]{\pnode(#1,1){#1} \pscircle*(#1,1){.1} \rput(#1,.5){$#2$}}
\def\vput#1{\cvput#1[#1]}

\newcommand{\edgebr}[1]{\makebox[3mm,linecolor=red]{#1}}

\numberwithin{equation}{section}
\title{New characterization of  two-state  normal distribution
\footnote{This research was sponsored by the Polish National Science Center grant No. 2012/05/B/ST1/00626.} }

\author{Wiktor Ejsmont}
\date{Department of Mathematics and Cybernetics, \\ Wroclaw University of Economics, \\
 Address: Komandorska 118/120, 53-345 Wroclaw, Poland,
\\
wiktor.ejsmont@gmial.com}
\maketitle

{\small
\textbf{Abstract}. In this article we give a purely noncommutative criterion for the characterization of
two-state  normal distribution. We prove that  families of two-state  normal distribution
can be described by relations which is similar to the conditional expectation in free probability, but has no classical
analogue. We also show a generalization of Bo\.zejko, Leinert and Speicher's formula from \cite{BoLR} (relating moments
and noncommutative cumulants).
} \\
\\
\newpage
\section{Introduction}
The original motivation for this paper is a desire to understand the  results about two state conditional expectation which were shown in  \cite{BoBr2}.    They   proved that if the 	
``outer'' state satisfies the condition  quadratic conditional variances,  then moment generating function satisfies some relation. An open problem in this area is a converse implication of their theorem. 
We will show that this relation  satisfies the condition from Theorem 2.1 of \cite{BoBr2}  and we extend this theorem (but that is not the main purpose of this article). 
The main goal of this paper is to construct a  new condition connected with the ``outer'' state which  gives us a new characterization of  two-state normal distribution.

The study of a random variable in conditionally free probability has been an active research field   during the last decade - see works \cite{An4,An1,An3,AnMlotko,BoWys,BoBr2,Hasebe,KrystekWoja1,KrystekWoja2,Krystek1,Mlotko2,Mlotko1,PopaWang,Wojakowski}. 
It is common in conditionally free probability, that its  properties, to a large extent, are analogous to those of the classical and free probability. 
The main aim of this paper is to produce a new characterization of the
  two state normal laws, which is close to the quadratic regression
property, but with no analog to the classical condition.
As an example, consider two random variables which have the same distribution (because the result is more transparent with this assumption). Suppose that $\mathbb{X}$, $\mathbb{Y}$  are $c$-free, self-adjoint, non-degenerate, centered and have the same distribution. Then  $\mathbb{X}$ and
$\mathbb{Y}$ have two-state normal laws (with respect to some states $(\varphi,\psi)$, which we will discuss later in section 2) if and only if  there exist constants $ a $, $ b $ 
such that
\begin{align} 
\varphi\big((\mathbb{X}-\mathbb{Y})^2\S^n\big)&= b \varphi(\S^n),\label{eq:wstepwzopr1}
\\ 
\varphi\big((\X-\Y)\S^n(\X-\Y)\big)&=\varphi\big([(1- b )\S^2+(-2 a + a  b )\S+( a ^2+ b ^2)\mathbb{I}]\S^n\big),
\label{eq:wstepwzopr2}
\end{align}
where $\S=\X+\Y$. This result is unexpected because in commutative and free
probability we have $$\varphi\big((\mathbb{X}-\mathbb{Y})^2\S^n\big)=\varphi\big((\X-\Y)\S^n(\X-\Y)\big),$$
for any classical and free variables $\mathbb{X}$ and $\mathbb{Y}$.
We also  show that equation \eqref{eq:wstepwzopr1} is equivalent with two different conditions.

At this point it is worth mentioning  about the characterization of type Laha-Lukacs in noncommutative and classical probability.  In \cite{LL,Wesolowski1}, all the classical random variables and processes of Meixner type using a quadratic regression property were  characterized. In free probability   Bo\.zejko, Bryc and Ejsmont  proved that the first conditional linear  moment and 
 conditional quadratic variances  characterize free Meixner laws (Bo\.zejko and Bryc \cite{BoBr}, Ejsmont \cite{Ejs}). 
Laha-Lukacs type  characterizations  of random variables in free probability are also studied by Szpojankowski,  Weso\l owski \cite{SzWes}.
 They give a characterization of noncommutative
free-Poisson and free-Binomial variables by properties of the first two conditional moments, 
which mimics Lukacs type assumptions known from classical probability. 
Similar results have been obtained in boolean probability by Anshelevich  \cite{An2}.  He showed that in the boolean theory the Laha-Lukacs property characterizes only the Bernoulli distributions. 
  It is worthwhile to mention the work of Bryc \cite{Br}, where the Laha-Lukacs property for $q$-Gaussian processes was shown. Bryc proved that classical processes corresponding to operators which satisfy a $q$-commutation relations, have linear regressions and quadratic conditional variances. 

The paper is organized as follows. In section 2 we review basic conditionally free probability, two-state normal laws  and  the statement of the main result. Next in the third section we quote complementary facts, lemmas and indications. In the fourth section we look more closely at non-crossing partitions with the first and last elements in the same block. In this section we also give extended version of Theorem 2.1 of \cite{BoBr2} (Theorem \ref{twr:7}) and generalize the Bo\.zejko, Leinert and Speicher's identity. Finally, in section 5 we  prove our main results.  

\section{Basic facts about two-state freeness condition}
Let $\mathcal{A}$ be a unital *-algebra with two-states $\varphi,\psi: \mathcal{A}\rightarrow \mathbb{C}$. We assume that states $\varphi$  fulfill the usual assumptions of positivity and normalization, and we assume the tracial property  $\psi(ab) =  \psi(ba)$ for $\psi$, but not
for $\varphi$.
A typical model of an algebra with two-states is a group algebra of a group
$G=*G_i$, where $*$ is a free product of groups $G_i$. Here $\varphi$ is the boolean product of the
individual states, the simplest example
is the free product of integers, $G_i = \mathbb{X}$, where $G_i$ is a free group with an arbitrary
number of generators, and $\varphi$ is the Haagerup state, $\phi(x) = r^{|x|}$, where $|x|$ is the
length of word $x\in G, -1\leq r\leq 1$, and state $\psi$  is $\delta_e$. For the details see \cite{Bozej1,Bozej2,BoBr2}.

A self-adjoint element $\X \in \mathcal{A}$ with moments that fulfill appropriate growth condition
defines a pair $(\mu, \nu)$ of probability measures on $\mathbb{R}$ such that 

 \begin{align} 
\varphi(\X^n)=\int_\mathbb{R} x^n  \mu(dx) \textrm{ and } \psi(\X^n)=\int_\mathbb{R} x^n  \nu(dx). 
\end{align}
We will refer to the measures  $\mu, \nu$  as the $\varphi$ -law and the  $\psi$-law of $\X$, respectively. In this paper we assume that  $\mu$ and $\nu$ are compactly supported probability measures, so moments  do not grow faster than exponentially.
 
\begin{defi}
Let $\pi= \{V_1, . . . , V_p\}$ be a partition of the linear ordered set
${1, \dots , n}$, i.e. the $V_i \neq \emptyset $ are ordered and disjoint sets whose union is $\{1, \dots , n\}$.
Then $\pi$ is called non-crossing if $a, c \inV_i$ and $b, d \inV_j$  with $a < b < c < d$ implies
$i = j$.

The sets $V_i\in \pi$ are called blocks. In a non-crossing partition ƒ$\pi$, a block $V_i$ is inner if for some $a, b \notin V_i$ (where $a$ and $b$ are  in  some other block of the partition $\pi$) and all
$x \in V_i , a <x< b$, otherwise it is called outer.
Family of all outer (resp. inner) blocks of $\pi$ will be denoted by $ Out(\pi)$ (resp. $Inn(\pi)$).
We will denote the set of all non-crossing partitions of the set $\{1, . . . , n\}$ by $NC(n)$.
\end{defi}

\begin{defi}
The free (non-crossing) cumulants are the $k$-linear maps $r_k : \mathcal{A}^k  \to\mathbb{C}$ ($r_k=r_k^\psi=r_k^\nu$) defined  by the recursive  formula (connecting them with mixed moments see \cite{NS})
\begin{align}
\psi(\mathbb{X}_{1}\mathbb{X}_{2}\dots \mathbb{X}_{n}) = \sum_{\nu \in NC(n)}r_{\nu}(\mathbb{X}_{1},\mathbb{X}_{2},\dots ,\mathbb{X}_{n}),\label{eq:DefinicjaKumulant}
\end{align}
where 
\begin{align}
r_{\nu}(\mathbb{X}_{1},\mathbb{X}_{2},\dots ,\mathbb{X}_{n}):=\Pi_{B \in \nu}r_{|B|}(\mathbb{X}_{i}:i \in B) ,
\end{align}
where $\mathbb{X}_{1},\mathbb{X}_{2},\dots ,\mathbb{X}_{n}\in \mathcal{A}.$
With each set of $\X_1, \dots , \X_n \in \mathcal{A}$ and a pair of states ($\varphi,\psi$) we associate the two-state free cumulants $R_k = R_{k}^{(\varphi,\psi)}= R_{k}^{(\mu,\nu)}$, $k =1, 2,\dots,$ which are  multilinear functions $R_k:\mathcal{A}^k:\rightarrow \mathbb{C}$ defined by
 \begin{align} 
\varphi(\X_1 \dots  \X_n)&= \sum_{k=1}^n  \sum_{s_{1}=1<s_2<\dots< s_k\leq n} R_k(\X_{s_1},\dots,\X_{s_k})\varphi(\X_{s_{k}+1}\dots \X_{s_{n}})\prod_{r=1}^{k-1} \psi(\prod_{j=s_r+1}^{s_{r+1}-1}\X_j). \label{eq:DefinicjaKumulant1}
\end{align}
The above equation is equivalent to 
 \begin{align} 
\varphi(\X_1 \dots  \X_n)&= \sum_{\nu \in NC(n)} \prod _{B \in Out(\nu)} R_{|B|}(\mathbb{X}_{i}: i\in B)
\prod _{B \in Inn(\nu)} r_{|B|}(\mathbb{X}_{i}: i\in B). \label{eq:DefinicjaKumulant2}
\end{align}
\end{defi}
Sometimes we will write $r_{k}(\mathbb{X})=r_{k}( \mathbb{X},\dots ,\mathbb{X} )$ and $R_{k}(\mathbb{X})=R_{k}( \mathbb{X},\dots ,\mathbb{X} )$. 
Fix $\X \in \mathcal{A}$ and consider the following 
power series
 \begin{align} 
r(z)=r_\nu(z)&=\sum_{i=0}^{\infty}r_{i+1}(\X,\dots,\X)z^{i}, \nonumber \\
R(z)=R_\mathbb{X}(z)=R_{(\mu,\nu)}(z)&=\sum_{i=0}^{\infty}R_{i+1}(\X,\dots,\X)z^{i}, \nonumber \\
M_\nu(z)&=\sum_{i=0}^{\infty}z^{i}\psi(\X^i), \nonumber \\
M_{\mu}(z)&=\sum_{i=0}^{\infty}z^{i}\varphi(\X^i) \nonumber.
\end{align}
For our purposes, the most convenient definition is the
following. The Cauchy-Stieltjes transform of $\mu$ can be expanded into the following formal power series
\begin{align}
G_\mu(z)=\int_{\mathbb{R}}\frac{1}{z-y}\mu(dy)=\sum_{n=0}^{\infty} m_n(\mu)\frac{1}{z^{n+1}}=\frac{1}{z}M_\mu\left(\frac{1}{z}\right),
\end{align}
where $m_n(\mu)$ is the $n$-th moment of $\mu$.
Bellow we introduce a definition of free independence  (see \cite{JasiulisKula,DelengaSniady,KulaWysocz,KrystekWoja1,KrystekWoja2,Lenczewski0,Lenczewski,Lenczewski1,Michna}). 
\begin{defi} (A) We say that subalgebras $\mathcal{A}_1,\mathcal{A}_2,\dots$ are
$\psi$-free if for every choice of $i_{1}\neq i_{2}\dots \neq i_{n}$ and every choice of $\X_i\in\mathcal{A}_i$ such that
 $\psi(\X_i) = 0$  we have
 \begin{align} 
\psi(\X_1\X_2\dots\X_n)=0.
\end{align}
(B) This family is $c$-freely independent if it is $\psi$-freely independent and, under the same
assumptions on $\X_1,\X_2,\dots,\X_n$ also
 \begin{align} 
\varphi(\X_1\X_2\dots\X_n)=\prod_{k=1}^n\varphi(\X_k).
\end{align} 
 \label{defi:DefinicjaNiezalDwustanowej}
\end{defi}
\noindent The above definition is equivalent to the following.
\begin{defi}
We say that subalgebras $\mathcal{A}_1,\mathcal{A}_2,\dots$ are $c$-free if for every
choice of $\X_1, \dots , \X_n\in \bigcup_j\mathcal{A}_j$ we have
\begin{align} 
r_n(\X_1,\dots,\X_n)=0 \textrm{ and }
R_n(\X_1,\dots,\X_n)=0 \textrm{ except if all $\X_j$ come from the same algebra.} \nonumber
\end{align}
\end{defi}

\begin{Rem}
It is important to note that Bo\.zejko and Bryc use the following definition of independence: 
\\ 
We say that subalgebras $\mathcal{A}_1,\mathcal{A}_2,\dots$ are ($\varphi,\psi$)-free if for every
choice of $\X_1, \dots , \X_n\in \bigcup_j\mathcal{A}_j$ we have
\begin{align} 
R_n(\X_1,\dots,\X_n)=0 \textrm{ except if all $\X_j$ come from the same algebra.}
\end{align}
It is important to note that ($\varphi,\psi$)-freeness is weaker than $c$-freeness. The $c$-freeness implies  ($\varphi,\psi$)-freeness  -- see Lemma 1.1 from \cite{BoBr2}.
\end{Rem}


\begin{defi}[$c$-free convolution]
Using free cumulants, we can define in a uniform way the free
convolution $\boxplus$ for example:
 \begin{align} 
r_n^{\nu_1\boxplus\nu_2}=r_n^{\nu_1}+r_n^{\nu_2}.
\end{align}

The two-state free (or conditionally free; $c$-free convolution; these terms will be used interchangeably)
convolution $\boxplus_c$ is an operation on pairs of measures, defined as follows: $(\mu_3,\nu_3)=(\mu_1,\nu_1)\boxplus_c(\mu_2,\nu_2)$ if and only if $\nu_3=\nu_1\boxplus\nu_2$ and
 \begin{align} 
R_n^{(\mu_3,\nu_3)}=R_n^{(\mu_1,\nu_1)}+R_n^{(\mu_2,\nu_2)}.
\end{align}
\end{defi}
\subsection{Two-state normal distribution and the main result}

Any probability measure $\mu$ on the real line,  all of whose moments are finite, has two
associated~sequences of Jacobi parameters $\alpha_i,\beta_i$ for example, $\mu$ is the spectral measure of the tridiagonal matrix

\begin{align}
\left(\begin{array}{c c c c c}
\alpha_0, & \beta_0, & 0, & 0, & \ddots  \\ 
1, & \alpha_1, & \beta_1, & 0, & \ddots
\\ 
0, & 1, & \alpha_2, & \beta_2, & \ddots
\\ 
0, & 0, & 1, & \alpha_3, &  \ddots
\\ 
\ddots & \ddots & \ddots & \ddots &  \ddots
\end{array}\right).
\end{align} 
We will denote this fact by
\begin{align}
J(\mu)=\left(\begin{array}{c c c c}
\alpha_0, & \alpha_1, & \alpha_2, & \dots  \\ 
\beta_0, & \beta_1, & \beta_2, & \dots
\end{array}\right)\label{eq:Jacobi1}
\end{align} 
with $\alpha_n(\mu):=\alpha_n,$ $\beta_n(\mu):=\beta_n$. These parameters are related to the moments of the
measure via the  Accardi-Bo\.zejko \cite{AcardiBozejko} formulas.
If the measure $\mu$ has all moments, then by a theorem of Stieltjes (see \cite{AkhGlaz}), it can be
expressed as a continued fraction:
\begin{align} 
G_\mu(z)=\cfrac{1}{z-\alpha_0 -\cfrac{\beta_0}{z-\alpha_1-\cfrac{\beta_1}{z-\alpha_2-\cfrac{\beta_2}{\ddots}}}}. \label{eq:CiaglaFrakcja}
\end{align} 

If some $\beta_i=0$ the continued fraction terminates, that is the subsequent  $\alpha$ and  $\beta$~coefficients
can be defined arbitrarily. See \cite{Chihara} for more details. The monic orthogonal polynomials $P_n$ for $\mu$  satisfy the following  recursion relation
\begin{align} 
xP_n(x) = P_{n+1}(x) +\alpha_nP_n(x) +\beta_{n-1}P_{n-1}(x),
\end{align} 
with $P_{-1}(x) = 0$.


\begin{defi}
$\X$ is a two-state  normal (Gaussian) distribution if $R_{k}(\X)=0$ and $r_{k}(\X)=0$  for $k>2$.  
\label{defi:DwustanowyNormalny}
\end{defi}
Without the loss  of generality we can assume (in this paper), that  two-state  normal element $\X$  has Jacobi parameters
\begin{align}
J(\mu_{ a , b })&=\left(\begin{array}{c c c c c }
 a , & 0, &0,& 0, & \dots  \\ 
 b , & 1, & 1,& 1 & \dots
\end{array}\right), \label{eq:JacobiDefiNoramal1} 
\\
J(\nu)&=\left(\begin{array}{c c c c c }
0, & 0, & 0,& 0, & \dots  \\ 
1, & 1, & 1,& 1 & \dots
\end{array}\right),\label{eq:JacobiDefiNoramal2} 
\end{align}
  which means that  $\mu_{a,b}$ (where $a\in \mathbb{R}$ and $b>0$ -- this assumption on $a$ and $b$ will be valid till the end of the work) is free Meixner distribution and $\nu$ is normalized Wigner's semicircle law. The first cumulants of this distribution are as follows  $R_1(\X)=a$, $R_2(\X)=b$, $r_1(\X)=0$ and $r_2(\X)=1$.   

For particular values of $ a $ and $ b $ the law of $\mu_{ a , b }$ is (see \cite{BoBr}):
\begin{itemize}
\item the Wigner's semicircle law if $ a =0$ and $ b =1$;
\item  the free Poisson  law if $ a  \neq 0$ and $ b  = 1$;
\item  the free Pascal (negative binomial) type law if $ b <1$ and $ a ^2>4(1- b )$;
\item  the free Gamma  law if $ b <1$ and $ a ^2=4(1- b )$;
\item the pure free Meixner  law if $ b <1$ and $ a ^2<4(1- b )$;
\item the free binomial  law  $ b >1$.
\end{itemize}

\subsection{The main result}

Now we can state the main result of this paper.
\begin{theo}
Suppose $\mathbb{X}$, $\mathbb{Y}$  are $c$-free, self-adjoint, non-degenerate, $\psi(\mathbb{X})=\psi(\mathbb{Y})=0$, $\psi(\mathbb{X}^2+\mathbb{Y}^2)=1$, $\varphi(\mathbb{X})=\alpha a ,\varphi(\mathbb{Y})=\beta a $ , $\varphi((\mathbb{X}+\mathbb{Y})^2)= b + a ^2$, $\beta R_k(\X)= \alpha R_k(\Y),\beta r_k(\X)= \alpha r_k(\Y) \textrm{ for some }\alpha,\beta>0,\alpha+\beta=1$  and all integers  $k\geq 1$. Let $\S=\X+\Y$, then the following  statements are equivalent:
\begin{enumerate}
\item  $\X$, $\Y$ have a two-state normal distribution, 
\item 
\begin{align} 
\varphi\big((\beta\mathbb{X}-\alpha\mathbb{Y})^2\S^n\big)&= \alpha\beta b \varphi\big(\S^n\big),
  \label{eq:PomocniczyNormalny2}
\\  
\varphi\big((\beta\X-\alpha\Y)\S^n(\beta\X-\alpha\Y)\big)&=\alpha\beta\varphi\big([(1- b )\S^2+(-2 a + a  b )\S+( a ^2+ b ^2)\mathbb{I}]\S^n\big),
\label{eq:PomocniczyNormalny}
\end{align}
\item  \eqref{eq:PomocniczyNormalny} and 
\begin{align} 
\varphi\big((\beta\X-\alpha\Y)\S^n(\beta\X-\alpha\Y)\big)&=
\alpha\beta b \psi\big(\S^n\big),
\label{eq:PomocniczyNormalny3}
\end{align}
\item  \eqref{eq:PomocniczyNormalny} and 
\begin{align} 
\varphi\big((\beta\X-\alpha\Y)(\X+\Y)(\beta\X-\alpha\Y)\S^n\big) =\alpha\beta a \varphi\big( (\beta\X-\alpha\Y)^2\S^n\big),
\label{eq:PomocniczyNormalny4}
\end{align}
\end{enumerate}
where  all above relations hold for all non-negative integers $n\geq 0.$

\label{twr:8}
\end{theo}
\begin{Rem}
In free probability  we can formulate the following theorem (see \cite{BoBr,Ejs}).

\begin{theo}
Suppose that $\mathbb{X}$, $\mathbb{Y}$  are free, self-adjoint, non-degenerate $\varphi(\mathbb{X})=\alpha a ,\varphi(\mathbb{Y})=\beta a $  and $\varphi(\mathbb{X}^2+\mathbb{Y}^2)=b+a^2$. Then $\mathbb{X}/\sqrt{\alpha}$ and  $\mathbb{Y}/\sqrt{\beta}$ have the free Meixner  laws $\mu_{a\sqrt{\alpha},b}$ and  $\mu_{a\sqrt{\beta},b}$, respectively, where $\alpha+\beta=1$, $a \in \mathbb{R},b>0$ if and only if
\begin{align} 
\varphi(\mathbb{X}|(\mathbb{X}+\mathbb{Y}))&={\alpha}(\mathbb{X}+\mathbb{Y}) +a\mathbb{I},\\
Var(\mathbb{X}|\mathbb{X}+\mathbb{Y}) &=\alpha\beta\big((1- b )(\mathbb{X}+\mathbb{Y})^2+(-2 a + a  b )(\mathbb{X}+\mathbb{Y})+( a ^2+ b ^2)\mathbb{I}\big)+a\mathbb{I} .  \label{eq:remofree}
\end{align} 
 We can easily show that equation \eqref{eq:remofree} is equivalent to $$\varphi((\beta\X-\alpha\Y)^2|\mathbb{X}+\mathbb{Y}) =\alpha\beta\big((1- b )(\mathbb{X}+\mathbb{Y})^2+(-2 a + a  b )(\mathbb{X}+\mathbb{Y})+( a ^2+ b ^2)\mathbb{I}\big),$$ so we see that  condition \eqref{eq:PomocniczyNormalny}  behaves like conditional variances in free probability. 
 \end{theo}

\end{Rem}
\section{Complementary facts, lemmas and indications}
\begin{defi}
\noindent We introduce the notation  
 \begin{align} 
\varphi_k(\X^n)= \sum_{j=k}^n  \sum_{s_{1}=1<s_2=2<\dots<s_k=k<\dots< s_j\leq n} R_j(\X)\varphi(\X^{s_{n}-1-s_{j}}) \prod_{r=k}^{j-1} \psi(\X^{s_{r+1}-1-s_{r}}), \label{eq:Kstalych}
\end{align}
where $k\leq n$. The above equation corresponds to moments of $\varphi$ with the first $k$ elements in the same block. Analogously we can define  $\varphi_k(\X_1 \dots  \X_n)$.
\end{defi}
\begin{exam}  \noindent For $k=3$ and $n=5$, we get:
\begin{align}
\varphi_3(\X^5)=R_3(\X)\varphi(\X^2)+R_4(\X)\varphi(\X)+R_4(\X)\psi(\X)+R_5(\X).
 \end{align}
\end{exam}
\noindent The following lemma is a two-state version of  Lemma 2.4 in \cite{Ejs2} (the proof is also similar).
\begin{lemm}
 Let $\X$ be a self-adjoint element of the algebra $\mathcal{A}$ then 
 \begin{align} 
\varphi_k(\X^{n+k})= \sum_{j=1}^{n} \psi(\X^{j-1}) \varphi_{k+1}(\X^{n+k-j+1})  + R_k(\X)\varphi(\X^{n-k}) ,\label{eq:RownanieKombinatoryczne}
\end{align}  
where $k,n\geq 1$ and we take convention $\psi(\X^0)=1$.
\label{lem:1} 
 \end{lemm}
\begin{proof} 
First, we will consider partitions with  the first $k$ elements in the same block, i.e. we sum only for $j=k$ in the equation \eqref{eq:Kstalych} which corresponds to 
$R_k(\X)\varphi(\X^{n-k}) $.  

On the other hand, for $j>k$ 
 denote $s(\nu)=min\{j:j>k,j\in B_1\}$ where $ B_1$ is the block 
which contains $1,\dots,k$ ( in Figure \ref{fig:FiguraExemple1} it is an element $j$). 
 This decomposes our situation into the $n$ classes which can be identified with the product $\psi(\X^{j-k-1})\times \varphi_{k+1}(\X^{n+2k-j+1})$ where $ j\in\{k+1, \dots , n+k\}$. 
Indeed, the blocks in which partitions are the elements $\{{k+1}, \dots ,{j-1}\}$ can be  identified with $\psi(\X^{j-k-1})$ (in Figure \ref{fig:FiguraExemple1} these are stars), and under the additional constraint that the first $k+1$ elements are in the same block,  the remaining blocks, which are partitions of the
set $\{1,\dots,k,j,{j+1}, . . . , {n+k}\}$, can be uniquely identified with  $\varphi_{k+1}(\X^{n+2k-j+1})$ (in Figure \ref{fig:FiguraExemple1} it is a part without the stars). This gives the formula \eqref{eq:RownanieKombinatoryczne} (if we re-index $j$) and proves the lemma.

\vspace{1 cm}
\begin{figure}[h]
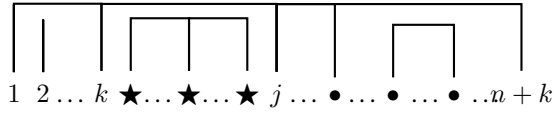

\begin{center}
\vspace{1cm}
\psset{nodesep=5pt}
\psset{angle=90}
\hspace{1mm}
\rnode{A}{\Node{$1$}} \rnode{A1}{\Node{$2$}}   \rnode{B}{\Node{\ldots}} \rnode{C1}{\Node{$k$}}
\rnode{D1}{\Node{$\bigstar$}} \rnode{D5}{\Node{\ldots}} \rnode{D2}{\Node{$\bigstar$}} \rnode{D5}{\Node{\ldots}}  \rnode{D4}{\Node{$\bigstar$}}   \rnode{E}{\Node{${j}$}} \rnode{E1}{\Node{\ldots}}  \rnode{E6}{\Node{$\bullet$}}  \rnode{E5}{\Node{\ldots}}  \rnode{E2}{\Node{$\bullet$}} \rnode{E3}{\ldots}\rnode{E7}{\Node{$\bullet$}}  \rnode{E3}{\ldots}  \rnode{F}{\Node{\textrm{ ${n+k}$ }}} 
\ncbar[armA=.7]{A1}{A1}
\ncbar[armA=.91]{A}{C1}
\ncbar[armA=.9]{A}{A}
\ncbar[armA=.89]{C1}{E}
\ncbar[armA=.9]{E}{F}
\ncbar[armA=.9]{E}{E6}
\ncbar[armA=.7]{D1}{D2}
\ncbar[armA=.5]{D3}{D3}
\ncbar[armA=.7]{D2}{D4}
\ncbar[armA=.7]{E2}{E7}
\caption{The main structure of non-crossing partitions of $\{1, 2,3 ,\dots,n+k\}$ with the first $k$ elements in the same block.}
\label{fig:FiguraExemple1}
\end{center}
\end{figure}
\end{proof}
\begin{defi}
Let $\mathbb{X}$ be a (self-adjoint) element of the algebra $\mathcal{A}$.  We introduce functions (series):
\begin{align} C^{(k)}_{\varphi,\psi}(z)=\sum_{n=0}^\infty \varphi_k(\X^{k+n}) z^{k+n}, \textrm{ where } k\geq 1 \end{align}
for sufficiently small $|z|$ and $z \in \mathbb{C}$. 
This series is convergent because we consider such a series as $M_{\nu_{\mathbb{X}}}(z)_\mathbb{X}$ and $M_{\mu_{\mathbb{X}}}(z)$ is convergent for sufficiently small $|z|$.   Thus from Lemma \ref{lem:1} we get that $C_{\varphi,\psi}^{(2)}$(z) is convergent because $C_{\varphi,\psi}^{(1)}(z)=M_{\mu_{\mathbb{X}}}(z)-1$. For $k>2$ this is immediate, by induction on $k$ and by using the
preceding lemma.
\end{defi}
\begin{lemm}
Let $\mathbb{X}$ be a (self-adjoint) element of the algebra $\mathcal{A}$,   then 
\begin{align}
C^{(k)}_{\varphi,\psi}(z)=M_\nu(z)C_{\varphi,\psi}^{(k+1)}(z)+R_k(\X)z^kM_\mu(z).
\end{align} 
\label{lem:2} 
 \end{lemm}
\begin{proof} It is clear from Lemma \ref{lem:1}  that we have
\begin{align} C^{(k)}_{\varphi,\psi}(z)&=\sum_{n=0}^\infty \varphi_k(\mathbb{X}^{k+n})z^{k+n}=\varphi_k(\mathbb{X}^{k})z^{k}+\sum_{n=1}^\infty \varphi_k(\mathbb{X}^{n+k})z^{k+n}\nonumber \\ &= \varphi_k(\mathbb{X}^{k})z^{k}+\sum_{n=1}^\infty [\sum_{i=0}^{n-1}\psi(\X^{i})\varphi_{k+1}(\mathbb{X}^{n+k-i})  +R_k(\X)\varphi(\mathbb{X}^{n})]z^{k+n}
 \nonumber \\
&= \varphi_k(\mathbb{X}^{k})z^{k}+\sum_{n=1}^\infty \sum_{i=0}^{n-1}\psi(\X^{i})z^{i}\varphi_{k+1}(\mathbb{X}^{n+k-i})z^{k+n-i} +R_k(\X)z^{k}\sum_{n=1}^\infty \varphi(\mathbb{X}^{n})z^{n}
 \nonumber \\
&= \sum_{n=1}^\infty \sum_{i=0}^{n-1}\psi(\X^{i})z^{i}\varphi_{k+1}(\mathbb{X}^{n+k-i})z^{k+n-i} +R_k(\X)z^{k}\sum_{n=0}^\infty \varphi(\mathbb{X}^{n})z^{n}
\nonumber \\ &=M_\nu(z)C^{(k+1)}_{\varphi,\psi}(z)+R_k(\X)z^kM_\mu(z),
 \end{align}
which proves the lemma. 
\end{proof}
\begin{exam}  \noindent For $k=1$, we get
\begin{align}
 C^{(1)}_{\varphi,\psi}(z)=M_\mu(z)-1=M_\nu(z)C^{(2)}_{\varphi,\psi}(z)+R_1(\X)zM_\mu(z). \label{eq:exemK=1} 
 \end{align}\\
\noindent  Similarly, by putting $k=2$, we obtain
\begin{align} 
C^{(2)}_{\varphi,\psi}(z)=M_\nu(z)C^{(3)}_{\varphi,\psi}(z)+R_2(\X)z^2M_\mu(z). \label{eq:exemK=2} 
 \end{align}\end{exam}
\begin{defi}
\noindent We introduce the notation  
 \begin{align} 
\varphi_{\shortparallel}(\X^n)= \sum_{k=1}^n  \sum_{s_{1}=1<s_2<\dots< s_k= n} R_k(\X^k)\prod_{r=1}^{k-1} \psi(\X^{s_{r+1}-1-s_r}),  \label{eq:PierwszyiOstatni}
\end{align}
where $n\geq 2$. The above equation corresponds to ``moments'' of $\varphi$ with the first and last element in the same block. 
\end{defi}
\begin{exam}  \noindent For  $n=5$, we get
\begin{align}
\varphi_{\shortparallel}(\X^5)=R_2(\X)\psi(\X^3)+2R_3(\X)\psi(\X^2)+R_3(\X)\psi^2(\X)+3R_4(\X)\psi(\X)+R_5(\X).
 \end{align}
\end{exam}
\begin{lemm}
 Let $\X$ be the self-adjoint element of the algebra $\mathcal{A}$, then 
 \begin{align} 
\varphi(\X^{n})= \sum_{j=2}^{n} \varphi_{\shortparallel}(\X^j) \varphi(\X^{n-j}) +R_1(\X) \varphi(\X^{n-1}),\label{eq:RownanieKombinatorycznePierwszyOstatni}
\end{align} 
where $n\geq 1$.
\label{lem:8} 
 \end{lemm}
\begin{proof} 
The proof is based on the analysis of the formula \eqref{eq:DefinicjaKumulant1}.
First, we will consider partitions with singleton 1, i.e.  $\pi= \{V_1, \dots,V_k\}$ where $V_1=\{1\}$. It is clear that the  sum over all  non-crossing partitions of this form corresponds to the term $R_1(\X) \varphi(\X^{n-1})$.

On the other hand, for such partitions   as in notation \eqref{eq:DefinicjaKumulant1}  let $s = s(\nu) \in \{2, 3, \dots , n\}$  denote the most-right element  of the block containing $1$.  This decomposes our situation into the $n-1$ classes $ s(\nu) = j, \textrm{ } j \in\{ 2, \dots , n\}$. This set can be identified with the product $\varphi_{\shortparallel}(\X^{j}) \varphi(\X^{n-j})$. 
Indeed, the blocks which partition the elements $\{ {j+1}, \dots ,{n} \}$  can be  identified with $\varphi(\X^{n-j})$ (on Figure \ref{fig:FiguraExemple2} it is the part on the right side of the element $j$), and under the additional constraint that is the first one (i.e. 1) and last element (i.e. $j$)  are in the same
block, the remaining blocks, which partition the set $\{1, \dots ,j\}$, can be
uniquely identified with $\varphi_{\shortparallel}(\X^{j})$, it follows from the definition of $\varphi_{\shortparallel}$, i.e. equation \eqref{eq:PierwszyiOstatni}   (in Figure \ref{fig:FiguraExemple2} it is the block  which contains 1 and $j$). This yields the formula \eqref{eq:RownanieKombinatorycznePierwszyOstatni} and proves the lemma. 
%
 
\begin{figure}[h]
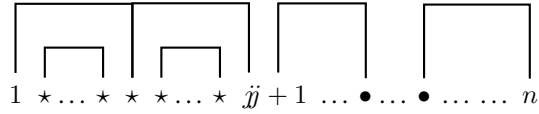

\begin{center}
\vspace{1cm}
\psset{nodesep=5pt}
\psset{angle=90}
\hspace{0.2mm}
\rnode{A}{\Node{$1$}} \rnode{A1}{\Node{$\star$}} \rnode{A1'}{\Node{\ldots}} \rnode{A1''}{\Node{$\star$}}  \rnode{A2}{\Node{$\star$}} \rnode{B}{\Node{$\star$}} \rnode{B1'}{\Node{\ldots}} \rnode{B1''}{\Node{$\star$}} \rnode{C1}{\Node{${j}$}}  \rnode{C1'}{\Node{${j+1}$}} \rnode{E}{\Node{ }}
   \rnode{E1}{\Node{\ldots}}  \rnode{E6}{\Node{$\bullet$}}  \rnode{E5}{\Node{\ldots}}  \rnode{E2}{\Node{$\bullet$}} \rnode{E3}{\ldots} \rnode{E3}{\ldots}  \rnode{F}{\Node{\textrm{ ${n}$ }}} 
\ncbar[armA=.92]{A}{A2}
\ncbar[armA=.4]{A1}{A1''}
\ncbar[armA=.4]{B}{B1''}
\ncbar[armA=.4]{B''''}{B''''}
\ncbar[armA=0.99]{A2}{C1}
\ncbar[armA=0.92]{C1'}{E6}
\ncbar[armA=1.5]{C}{C}
\ncbar[armA=.7]{D2}{D4}
\ncbar[armA=.5]{E7}{E7}
\ncbar[armA=.98]{E2}{F}
\caption{The main structure of the non-crossing partitions of $\{1, 2,3 ,\dots,j,\dots,n\}$ with the first one and the $j$-th element in the same block.}
\label{fig:FiguraExemple2}
\end{center}
\end{figure}
\end{proof}
\begin{defi}
Let $\mathbb{X}$ be a (self-adjoint) element of the algebra $\mathcal{A}$.  We introduce the following functions (series):
\begin{align} C_{{\shortparallel}}(z)= C_{{\shortparallel},\varphi,\psi}(z)=\sum_{n=0}^\infty \varphi_{\shortparallel}(\X^{2+n}) z^{2+n}, \end{align} 
for sufficiently small $z$ and $z \in \mathbb{C}$. This series is convergent by a similar argument as for $C^{(k)}_{\varphi,\psi}(z)$ (using Lemma \ref{lem:8}).  
\end{defi}

\begin{lemm}
Let $\mathbb{X}$ be a (self-adjoint) element of the algebra $\mathcal{A}$,   then 
\begin{align}
M_\mu(z)-1=M_\mu(z)C_{\shortparallel}(z)+R_1(\X)zM_\mu(z). \label{eq:TransformatyKombinatorycznePierwszyOstatni}
\end{align} 
\label{lem:4} 
 \end{lemm}
\begin{proof}  The proof is analogous to the proof of  Lemma \ref{lem:2} and based on the equation \eqref{eq:RownanieKombinatorycznePierwszyOstatni}. 
\end{proof}
\begin{cor} We have the following identity
\begin{align} 
M_\nu(z)C^{(2)}_{\varphi,\psi}(z)=M_\mu(z)C_{\shortparallel}(z).
\end{align} 
\end{cor}
\begin{proof}  Compare the equation \eqref{eq:exemK=1}  with the equation \eqref{eq:TransformatyKombinatorycznePierwszyOstatni}.
\end{proof}

\begin{lemm}
Suppose that $\X$, $\Y$ are self-adjoint, c-free, $\beta R_k(\X)= \alpha R_k(\Y) \textrm{ for some }\alpha,\beta>0$ and all integer  $k\geq 1$. Then 
\begin{align} R_{k}(\beta\mathbb{X}-\alpha\mathbb{Y},\mathbb{X+Y},\mathbb{X+Y},\dots,\mathbb{X+Y},\mathbb{X+Y})=0,\label{eq:zalozeniekumulantyRozdzial4}
\end{align}
\begin{align} R_{k}(\beta\mathbb{X}-\alpha\mathbb{Y},\beta\mathbb{X}-\alpha\mathbb{Y},\mathbb{X+Y},\dots,\mathbb{X+Y},\mathbb{X+Y})=\alpha\beta R_{k}(\X+\Y),\label{eq:zalozeniekumulanty2Rozdzial4}
\end{align}
\begin{align} R_{k}(\beta\mathbb{X}-\alpha\mathbb{Y},\mathbb{X+Y},\mathbb{X+Y},\dots,\mathbb{X+Y},\beta\mathbb{X}-\alpha\mathbb{Y})=\alpha\beta R_{k}(\X+\Y).\label{eq:zalozeniekumulanty3Rozdzial4}
\end{align}
\label{lem:3}
\end{lemm} 
\begin{proof} By taking into account the fact that $R_k$ are multilinear functions and  from the assumption of  $c$-free and $\beta R_k(\X)= \alpha R_k(\Y)$ we get

\begin{align} R_{k}(\beta\mathbb{X}-\alpha\mathbb{Y},\mathbb{X+Y},\mathbb{X+Y},\dots,\mathbb{X+Y})=\beta R_{k}(\X)-\alpha R_{k}(\Y)=0,  
\end{align}
and similarly for $k \geqslant 2$ 
\begin{align} 
&R_{k}(\beta\mathbb{X}-\alpha\mathbb{Y},\beta\mathbb{X}-\alpha\mathbb{Y},\mathbb{X+Y},\dots,\mathbb{X+Y})\nonumber\\&=\beta^2 R_{k}(\X)+\alpha^2 R_{k}(\Y) =\beta \alpha R_{k}(\X+\Y), 
\end{align}
\begin{align} &R_{k}(\beta\mathbb{X}-\alpha\mathbb{Y},\mathbb{X+Y},\dots,\mathbb{X+Y},\beta\mathbb{X}-\alpha\mathbb{Y})\nonumber\\&=\beta^2 R_{k}(\X)+\alpha^2 R_{k}(\Y) =\beta \alpha R_{k}(\X+\Y).
\end{align}
and the assertion
follows.
\end{proof}

\begin{lemm}
Suppose that $\X$, $\Y$ are self-adjoint, c-free, $\beta R_k(\X)= \alpha R_k(\Y) \textrm{ for some }\alpha,\beta>0$ and all integer  $k\geq 1$. Then 
\begin{enumerate}
\item $\varphi((\beta\mathbb{X}-\alpha\mathbb{Y})^2(\mathbb{X+Y})^n) = \alpha\beta \varphi_2((\mathbb{X+Y})^{n+2})$,
\item $\varphi\big((\beta\mathbb{X}-\alpha\mathbb{Y})(\mathbb{X+Y})^n(\beta\mathbb{X}-\alpha\mathbb{Y})\big) = \alpha\beta \varphi_{\shortparallel}((\mathbb{X+Y})^{n+2})$.
\end{enumerate}
\label{lemm:wlasnosciFunkcjiFi}
\end{lemm}
\begin{proof}
\noindent 1. Now we use the moment-cumulant formula (\ref{eq:DefinicjaKumulant}) and \eqref{eq:zalozeniekumulanty2Rozdzial4}
 \begin{align}
\varphi((\beta\mathbb{X}-\alpha\mathbb{Y})^2(\mathbb{X+Y})^n) =
\varphi_2((\beta\mathbb{X}-\alpha\mathbb{Y})(\beta\mathbb{X}-\alpha\mathbb{Y})(\mathbb{X+Y})^n) =
\alpha\beta \varphi_2((\mathbb{X+Y})^{n+2}), \label{eq:DwieSpojneKumulantyRozdzial4}
 \end{align}
because if the first element $\beta\mathbb{X}-\alpha\mathbb{Y}$ is in the partition with an element only from the ``part'' $(\mathbb{X+Y})^n$  then we have  \eqref{eq:zalozeniekumulantyRozdzial4} (the sum over this partition vanishes). Thus the first element and second one must be in the same block  and taking into account the equation  \eqref{eq:zalozeniekumulanty2Rozdzial4}, we get \eqref{eq:DwieSpojneKumulantyRozdzial4}.
\\
\\
\noindent 2. Now we show \begin{align} 
\varphi((\beta\mathbb{X}-\alpha\mathbb{Y})(\mathbb{X}+\mathbb{Y})^n(\beta\mathbb{X}-\alpha\mathbb{Y}))=\alpha\beta\varphi_{\shortparallel}((\mathbb{X}+\mathbb{Y})^{n+2}).
\label{eq:Pom2Propozycja1}
\end{align}
 We have that either the first and
the last elements are in different blocks, or they are in the same block. In the first case,
$\varphi((\beta\mathbb{X}-\alpha\mathbb{Y})(\mathbb{X}+\mathbb{Y})^n(\beta\mathbb{X}-\alpha\mathbb{Y}))=0$
 by the equation \eqref{eq:zalozeniekumulantyRozdzial4}. On the other hand, if they are in the
same block, then  from the Lemma \ref{lem:3}  we get \eqref{eq:Pom2Propozycja1}. 

\end{proof}

Now we present  a theorem which follows from the main result of \cite{AnMlotko}. It will be used in the proof of the main theorem in order to calculate the moment generating function of free convolution.
\begin{theo}
Let $(\mu,\nu)$ be a pair of measures with Jacobi parameters \eqref{eq:JacobiDefiNoramal1} and \eqref{eq:JacobiDefiNoramal2}, respectively. Then the conditionally free power $(\mu_t,\nu_t)=(\mu,\nu)^{\boxplus_ct}$ exists for $t\geq 0$ and we have 
\begin{align}
J(\mu_t)=\left(\begin{array}{c c c c c }
 a t, & 0, & 0,& 0, & \dots  \\ 
 b t, & t, & t,& t & \dots
\end{array}\right)\label{eq:JacobiDefiMeixner3}
\end{align}
and
\begin{align}
J(\nu_t)=\left(\begin{array}{c c c c c }
0, &0, & 0,& 0, & \dots  \\ 
t, & t, & t,& t & \dots
\end{array}\right).\label{eq:JacobiDefiMeixner4} 
\end{align}
\label{twr:1}
\end{theo}
\begin{lemm}[]
Suppose $\mathbb{X}$, $\mathbb{Y}$  are $c$-free and self-adjoint. Denote  the distribution of $\mathbb{X}/\sqrt{\alpha}$ and $\mathbb{Y}/\sqrt{\beta}$ by  $(\mu_1, \nu_1)$ and $(\mu_2, \nu_2)$,   respectively.  We assume that Jacobi parameters for this measure are equal to 
\begin{align}
J(\mu_1)&=\left(\begin{array}{c c c c c }
 a \sqrt{\alpha}, & 0, & 0,& 0, & \dots  \\ 
 b , & 1, & 1,& 1 & \dots
\end{array}\right) \textrm{, }J(\mu_2)=\left(\begin{array}{c c c c c }
 a \sqrt{\beta}, & 0, & 0,& 0, & \dots  \\ 
 b , & 1, & 1,& 1 & \dots
\end{array}\right)  %
\end{align} 
\textrm{and}\\
\begin{align}
J(\nu_1)&=\left(\begin{array}{c c c c c }
0, & 0, & 0,& 0, & \dots  \\ 
1, & 1, & 1,& 1 & \dots
\end{array}\right) 
\textrm{, }\textrm{ }\textrm{ }\textrm{ }\textrm{ }\textrm{ }\textrm{ }\textrm{ }J(\nu_2)=\left(\begin{array}{c c c c c }
0, & 0, & 0,& 0, & \dots  \\ 
1, & 1, & 1,& 1 & \dots
\end{array}\right). 
\end{align}
where $\alpha,\beta > 0 $, $\alpha+\beta =1 $.  Then   
 $\X+\Y$ has two-state normal law $(\mu,\nu)$ with Jacobi parameters
\eqref{eq:JacobiDefiNoramal1} and \eqref{eq:JacobiDefiNoramal2}, respectively . 
\label{lem:5}
\end{lemm}
\begin{proof}
We use the following well-known fact if a certain 
 variable $\mathbb{X}$ has distribution with Jacobi parameters 
\eqref{eq:Jacobi1} then $\gamma\mathbb{X}$ (where $\gamma\in \mathbb{R}$)  has the following   Jacobi parameters (see \cite{HorObata})
\begin{align}
\left(\begin{array}{c c c c}
\gamma\alpha_0, & \gamma\alpha_1, & \gamma\alpha_2, & \dots  \\ 
\gamma^2\beta_0, & \gamma^2\beta_1, & \gamma^2\beta_2, & \dots
\end{array}\right).\label{eq:Jacobi7}
\end{align} 
Thus we deduce that $\mathbb{X}$ and $\mathbb{Y}$ has respectively the following Jacobi parameters (with respect to the state $\varphi$)
\begin{align}
\left(\begin{array}{c c c c c }
 a \alpha, & 0, & 0,& 0, & \dots  \\ 
 b \alpha, & \alpha, & \alpha,& \alpha & \dots
\end{array}\right) \textrm{, } 
\left(\begin{array}{c c c c c }
 a \beta, & 0, & 0,& 0, & \dots  \\ 
 b \beta, & \beta, & \beta,& \beta & \dots
\end{array}\right). \nonumber 
\end{align}
Using Theorem \ref{twr:1} we deduce that the law of $\mathbb{X}+\mathbb{Y}$ is given by \eqref{eq:JacobiDefiNoramal1} with respect to the state $\varphi$. Analogously,~we have that $\mathbb{X}+\mathbb{Y}$ has Jacobi parameters \eqref{eq:JacobiDefiNoramal2} with respect to the state $\psi.$
\end{proof}

\section{A new relation  in conditionally free probability}
\subsection{A generalization of Bo\.zejko, Leinert and Speicher's  identity}
From \cite{BoLR}, we have  the following relation
\begin{align}  
M_\mu(z)\big(1-z\mathcal{R}_{\mathbb{X}}(zM_\nu(z))\big)=1 . \label{eq:Speicher}
\end{align}
The relation \eqref{eq:Speicher} can be generalized as follows:
\begin{prop}
Suppose that $\mathbb{X}$ is a self-adjoint element of the algebra $\mathcal{A}$ 
, then 
\begin{align}  
C_{\varphi,\psi}^{(k)}(z)=\mathcal{R}_{\mathbb{X}}^{(k)}(zM_\nu(z))z^kM_\mu(z),
\end{align} 
where $\mathcal{R}_{\mathbb{X}}^{(k)}(z)=\sum_{i=k}^{\infty}R_{i}(\mathbb{X})z^{i-k}$. 
\end{prop} 
\begin{proof}
We prove this by the induction on $k$. The case $k = 1$ is clear because $C_\mu^{(1)}(z)=M_\mu(z)-1$. 
The induction step $k \Rightarrow k+1 $ (for $k>1$) follows immediately  using Lemma \ref{lem:2} which gives
\begin{align}  
C^{(k+1)}_{\varphi,\psi}(z)=\frac{C^{(k)}_{\varphi,\psi}(z)}{M_\mu(z)}-R_k(\X)z^k&=\mathcal{R}_{\mathbb{X}}^{(k)}(zM_\nu(z))z^k-R_k(\X)z^k 
\\&=\mathcal{R}_{\mathbb{X}}^{(k+1)}(zM_\nu(z))z^{k+1}M_\mu(z). 
\end{align}  
\end{proof}
\subsection{A new relation between moments }
In this subsection we explain the motivation for introducing $\varphi_{\shortparallel}$ from the
point of view of two-state free probability. The answer to this problem turns out to be the following:
the relation between  measures whose Jacobi parameters are described by \eqref{eq:Jacobi1} and  other measure whose Jacobi parameter  equals 
 \begin{align}
\left(\begin{array}{c c c c}
\alpha_1, & \alpha_2, & \alpha_3, & \dots  \\ 
\beta_1, & \beta_2, & \beta_3, & \dots
\end{array}\right),\label{eq:Jacobi2}
\end{align} 
is contained in  $\varphi_{\shortparallel}$. 
\begin{theo}
Suppose that $\mathbb{X}$ and $\mathbb{Y}$ are self-adjoint elements of algebra $\mathcal{A}$. 
Denote by $\mu$ the distribution of 
$\X$ with respect to $\varphi$, and by $\rho$ the distribution of 
$\Y$ with respect to  $\varphi$ (the distribution of the state $\psi$ is irrelevant in this theorem). If measure $\mu$ has Jacobi parameters described by \eqref{eq:Jacobi1} where $\beta_0>0$, 
then the relation between the measure $\rho$ of the  variable $\mathbb{Y}$ described by the
parameter \eqref{eq:Jacobi2} is given by
\begin{align} 
\varphi_{\shortparallel}(\X^{n+2})=\beta_0\varphi(\Y^n)
, \label{eq:TwierdzenieRozdzial4Rozklady}
\end{align} 
for all $n\geqslant 0$.
\label{twr:5}
 \end{theo}
\begin{proof}
From \eqref{eq:CiaglaFrakcja} we have 
\begin{align} 
G_\mu(z)=\frac{1}{z-\alpha_0-\beta_0G_\rho(z)}.
\end{align} 
By using the relations $M_\mu(z)=\frac{1}{z}G_\mu(\frac{1}{z})$ and  $M_\rho(z)=\frac{1}{z}G_\rho(\frac{1}{z})$ we see that
\begin{align} 
M_\mu(z)({1-z\alpha_0-\beta_0z^2M_\rho(z)})=1. \label{eq:Rozdzial4pom1}
\end{align} 
Applying Lemma  \ref{lem:4} we get
\begin{align}
M_\mu(z)-1=M_\mu(z)C_{\shortparallel}(z)+\alpha_0zM_\mu(z). \label{eq:Rozdzial4pom2}
\end{align} 
where $C_{\shortparallel}(z)$ is the function for  $\X$. Now we substitute \eqref{eq:Rozdzial4pom2} to the equation \eqref{eq:Rozdzial4pom1}  and after a simple  computation, we obtain
\begin{align} 
\beta_0z^2M_\rho(z)=C_{\shortparallel}(z), \label{eq:Rozdzial4pom3}
\end{align}
which is equivalent to \eqref{eq:TwierdzenieRozdzial4Rozklady} and this completes the proof.
\end{proof}

\begin{cor}
An important benefit of the above theorem is the following one if we know  $\varphi(\X)$, $\varphi(\X^2)$  and $C_{\shortparallel}(z)$ then we also know the measure $\mu$. This result will be applied in the proof of the main theorem. \label{cor:Wniosek1}
 \end{cor}
\begin{cor}
 If $\beta_0=1$ then $\varphi_{\shortparallel}(\X^{n+2})$  is the moment of the variable described by Jacobi parameters \eqref{eq:Jacobi2}. \label{cor:Wniosek2}
 \end{cor}

\subsection{Some consequences for a two-state normal distribution}

\begin{prop}
Suppose that we have a 
self-adjoint variable  $\X$ with mean $\varphi(\X)= a $ and second moment  $\varphi(\X^2)= b + a ^2$. 
Then the following statements are equivalent:
\begin{enumerate}
\item the Jacobi parameter is equal to
\begin{align}
J(\varphi)=J(\mu)=\left(\begin{array}{c c c c c }
 a , & 0, & 0 ,& 0, & \dots  \\ 
 b , & 1, & 1,& 1 & \dots
\end{array}\right)\label{eq:Jacobi3}, 
\end{align}   
\item the following equation 
 \begin{align}
 \varphi_{\shortparallel}(\X^{n+2})= b \int x^n \rho(dx), \label{eq:WzorWniosekzTwierdzenieRozdz3}
\end{align}
is satisfied for all $n\geqslant 0$ and $\rho$ is the  Wigner's semicircle law with mean 0 and variance 1,

\item $C_{\shortparallel}(z)$ satisfies the equation
\begin{align}
(C_{\shortparallel}(z))^2- b C_{\shortparallel}(z)+ b ^2z^2=0,
\label{eq:MomentyGenerujaceFunkcjaLemat}
\end{align}
\item the relation 
\begin{align}
M_\mu(z) \big( b ^2z^2- b (1-z a )+(1-z a )^2\big)-1+z a + b =C_{\shortparallel}(z),\label{eq:PierwszyOstatniorazVarphi}
\end{align}
is true.
\end{enumerate}
\label{lem:6}
\end{prop}

\begin{proof} 
1 $\Rightarrow$ 2. 
 By Theorem \ref{twr:5} we have that $\varphi_{\shortparallel}(\X^{n+2})/ b $ is the moment of measure described by Jacobi parameters 
\begin{align}
\left(\begin{array}{c c c c c }
 0, & 0,& 0, & 0, &\dots  \\ 
 1, & 1,& 1, & 1, & \dots
\end{array}\right),\label{eq:Jacobi4} 
\end{align}
which gives that this is Wigner's semicircle law.
\\
2 $\Rightarrow$ 3. 
The moment generating function of measure $\rho$ given by the Jacobi parameters \eqref{eq:Jacobi4} satisfies (see \cite{Ejs}) the  equation 
\begin{align}
M_\rho^2(z)z^2-M_\rho(z)+1=0. \label{eq:MomentyGenerujaceFunkcja}
\end{align}
The equation \eqref{eq:WzorWniosekzTwierdzenieRozdz3} is equivalent to $C_{\shortparallel}(z)= b z^2M_\rho(z)$ and substituting this  to the equation \eqref{eq:MomentyGenerujaceFunkcja}  we get \eqref{eq:MomentyGenerujaceFunkcjaLemat}. 
\\3 $\Rightarrow$ 4.
If we now apply Lemma \ref{lem:4} i.e.  $M_\mu(z)C_{\shortparallel}(z)=M_\mu(z)-1-z a M_\mu(z)$ to the equation \eqref{eq:MomentyGenerujaceFunkcjaLemat} we get
\begin{align}
(M_\mu(z)-1-z a M_\mu(z))C_{\shortparallel}(z) 
 - b (M_\mu(z)-1-z a M_\mu(z))+ b ^2z^2M_\mu(z)=0, \label{eq:MomentyGenerujaceFunkcja2}
\end{align}
and then we  use $M_\mu(z)-1-z a M_\mu(z)=C_{\shortparallel}(z)M_\mu(z)$ again and after simple computation we get \eqref{eq:PierwszyOstatniorazVarphi}.
\medskip
%
\\ \noindent
4 $\Rightarrow$ 1. As we have seen in the  explanation above, each of the  steps above are equivalent so from \eqref{eq:MomentyGenerujaceFunkcja2} we get \eqref{eq:WzorWniosekzTwierdzenieRozdz3} and taking into account that $\varphi(\X)= a $,   $\varphi(\X)= b + a ^2$ and Corollary \ref{cor:Wniosek1} we get that Jacobi parameters for $\mu$ are equal to \eqref{eq:Jacobi3}.

\end{proof}

\subsection{Regression characterization for two-state algebras }
The original motivation for this paper is a desire to understand the Theorem 2.1 in the work of Bo\.zejko and Bryc \cite{BoBr2}.
They proved, that if we have a two-state variable with the same distribution and we assume a two-state analog  of conditional quadratic  variances  (equation \eqref{eq:TwierdzenieRozkladyRozdz3}, below), then the corresponding moment generating function satisfies some relationship i.e. \eqref{eq:TwierdzenieRelacjaSzeregi3}.
An open problem in this area is the converse implication to their theorem. In the theorem below we present a stronger version of Theorem 2.1 from the work \cite{BoBr2} and the converse implication.

\begin{theo}
Suppose $\X$, $\Y$ are self-adjoint, $c$-free, $\varphi(\mathbb{X}+\mathbb{Y})= a $, $\varphi((\mathbb{X}+\mathbb{Y})^2)= b + a ^2$,  $\beta R_k(\X)= \alpha R_k(\Y) \textrm{ for some }\alpha,\beta>0$, $\alpha+\beta=1$  and all integers  $k\geq 1$. Let $\S=\mathbb{X}+\mathbb{Y}$, then the following  statements are equivalent:

\begin{enumerate}
\item state $\varphi$ is connected with ``conditional quadratic variances'' by the relation  
\begin{align} 
\varphi\big((\beta\mathbb{X}-\alpha\mathbb{Y})^2
\S^n\big)=\frac{\alpha\beta }{\tilde{b}+1}\varphi\big((\tilde{b}\S^{2}+(\tilde{a}- a \tilde{b})\S+\mathbb{I}( b - a \tilde{a}))\S^n \big), 
\label{eq:TwierdzenieRozkladyRozdz3}
\end{align} 

\item   the relation
between the moment generating function  $M_\mu(z)$ and $M_\nu(z)$  is given
by
\begin{align} 
M_\mu(z)=\frac{(\tilde{b}+z\tilde{a})M_\nu(z)-\tilde{b}-1}{M_\nu(z)[( b - a \tilde{a})z^2+z(\tilde{a}- a \tilde{b})+\tilde{b}]-(\tilde{b}+1)(1- a z)}, \label{eq:TwierdzenieRelacjaSzeregi3}
\end{align}
\item  the following relation
between   $C_{\shortparallel}(z)$ and $M_\nu(z)$ 
\begin{align} 
C_{\shortparallel}(z)=\frac{M_\nu(z)z^2 b }{-M_\nu(z)(\tilde{b}+z\tilde{a})+\tilde{b}+1},\label{eq:TwierdzenieRozkladyRozdz3fipsi}
\end{align}
where $a,b,\tilde{a},\tilde{b} \in \mathbb{R} $, $\tilde{b} > -1$ and $b> 0$ is satisfied.
\end{enumerate}

\label{twr:7}
\end{theo}

\begin{proof}
1 $\Rightarrow$ 2: Suppose, that the equality \eqref{eq:TwierdzenieRozkladyRozdz3} holds.  
Thus from \eqref{eq:TwierdzenieRozkladyRozdz3} and Lemma \ref{lemm:wlasnosciFunkcjiFi} we get 

 \begin{align}
\alpha\beta \varphi_2(\S^{n+2})=\frac{\alpha\beta }{\tilde{b}+1}\varphi((\tilde{b}\S^{2}+(\tilde{a}- a \tilde{b})\S+\mathbb{I}( b - a \tilde{a}) )\S^n) .
 \end{align}
A routine argument  relates now the power series 
 \begin{align}
(\tilde{b}+1)C^{(2)}_{\varphi,\psi}(z)=\tilde{b}M_\mu(z)-\tilde{b} a z-\tilde{b}+z((\tilde{a}- a \tilde{b})M_\mu(z)-(\tilde{a}- a \tilde{b}))+z^2( b - a \tilde{a}) M_\mu(z)
\nonumber \\ =M_\mu(z)(( b - a \tilde{a})z^2+z(\tilde{a}- a \tilde{b})+\tilde{b})-\tilde{b}-z\tilde{a}. \label{eq:pom3rozdz3}
 \end{align}
If in \eqref{eq:pom3rozdz3} we multiply both sides by $M_\nu(z)$ and use the fact \eqref{eq:exemK=1} with $R_1(\X + \Y) =  a $, we get
 \begin{align}
(\tilde{b}+1)(M_\mu(z)-1- a zM_\mu(z))=M_\nu(z)(M_\mu(z)(( b - a \tilde{a})z^2+z(\tilde{a}- a \tilde{b})+\tilde{b})-\tilde{b}-z\tilde{a}), \label{eq:pom1rozdz3}
 \end{align}
or equivalently
 \begin{align}
M_\mu(z)=\frac{(\tilde{b}+z\tilde{a})M_\nu(z)-\tilde{b}-1}{M_\nu(z)[( b - a \tilde{a})z^2+z(\tilde{a}- a \tilde{b})+\tilde{b}]-(\tilde{b}+1)(1- a z)}. \label{eq:pom2rozdz3}
 \end{align}

\noindent 2 $\Rightarrow$ 3: If we use the formula \eqref{eq:TransformatyKombinatorycznePierwszyOstatni} with $R_1(X + Y) =  a $ to the equation \eqref{eq:pom1rozdz3} we obtain
 \begin{align}
(\tilde{b}+1)C_{\shortparallel}(z)=M_\nu(z)(( b - a \tilde{a})z^2+z(\tilde{a}- a \tilde{b})+\tilde{b})-(\tilde{b}+z\tilde{a})(1-C_{\shortparallel}(z)-z a )M_\nu(z), \label{eq:pom2rozdz3}
 \end{align}
or equivalently

\begin{align} 
C_{\shortparallel}(z)=\frac{M_\nu(z)z^2 b }{-M_\nu(z)(\tilde{b}+z\tilde{a})+\tilde{b}+1}.
\end{align}
3 $\Rightarrow$ 1: Suppose now, that equality \eqref{eq:TwierdzenieRozkladyRozdz3fipsi} holds. Applying \eqref{eq:exemK=1} to \eqref{eq:pom2rozdz3} we obtain (\ref{eq:pom1rozdz3}). Dividing \eqref{eq:pom1rozdz3} by $M_\nu(z)$ and applying \eqref{eq:exemK=1} we obtain (\ref{eq:pom3rozdz3}), which is equivalent to \eqref{eq:TwierdzenieRozkladyRozdz3}.
 \end{proof}

\begin{prop}
Suppose $\X$, $\Y$ are self-adjoint, $c$-free, $\varphi(\mathbb{X}+\mathbb{Y})= a $, $\varphi((\mathbb{X}+\mathbb{Y})^2)= b + a ^2$,  $\beta R_k(\X)= \alpha R_k(\Y) \textrm{ for some }\alpha,\beta>0$, $\alpha+\beta=1$  and all integers  $k\geq 1$. Then the following three statements  are equivalent:
\newline
\begin{align} 
\textrm{ (1) } &
\varphi((\beta\mathbb{X}-\alpha\mathbb{Y})^2\S^{n})=\alpha\beta b \varphi(\S^{n}),
\label{eq:TwierdzenieRozkladyRozdzPropo3}
\\
\textrm{ (2) } &\varphi((\beta\mathbb{X}-\alpha\mathbb{Y})\S^{n}(\beta\mathbb{X}-\alpha\mathbb{Y}))=\alpha\beta b \psi(\S^{n}),
\label{eq:TwierdzenieRozkladyRozdz3fipsiPropo}
\\
%
\textrm{ (3) } &
\varphi((\beta\mathbb{X}-\alpha\mathbb{Y})(\X+\Y)(\beta\mathbb{X}-\alpha\mathbb{Y})\S^{n})=\alpha\beta a \varphi((\beta\mathbb{X}-\alpha\mathbb{Y})^2\S^{n}),
\label{eq:TwierdzenieRozkladyRozdz3fipsiPropoPomoc}
\end{align}
for all non-negative integers $n\geq 0$, where $\S=\X+\Y$. 
\label{twr:9}
\end{prop}
\begin{proof}
\noindent (1 $\Rightarrow$ 2, 1 $\Leftarrow $ 2): If we put $\tilde{a}=0$ and $\tilde{b}=0$ in the equation  \eqref{eq:TwierdzenieRozkladyRozdz3}, we obtain the
 relation \eqref{eq:TwierdzenieRozkladyRozdzPropo3} which by Theorem \ref{twr:7} is equivalent to 
\begin{align} 
C_{\shortparallel}(z)= b M_\nu(z)z^2. \label{eq:Pom1Propozycja1}
\end{align}
From \eqref{eq:Pom1Propozycja1} we see that  $\varphi_{\shortparallel}(\S^{n+2})= b \psi(\S^{n})$. Using  Lemma \ref{lemm:wlasnosciFunkcjiFi} we get $$\varphi((\beta\mathbb{X}-\alpha\mathbb{Y})\S^{n}(\beta\mathbb{X}-\alpha\mathbb{Y}))=\alpha\beta b \psi(\S^{n}).$$ 
\smallskip

\noindent
(1 $\Rightarrow$ 3, 1 $\Leftarrow $ 3):
We consider the expression $\varphi((\beta\mathbb{X}-\alpha\mathbb{Y})(\X+\Y)(\beta\mathbb{X}-\alpha\mathbb{Y})(\mathbb{X+Y})^n)$. Since  we have that either the first element and
the third one are in different blocks, or they are in the same block. In the first case the
sum  vanishes by  \eqref{eq:zalozeniekumulantyRozdzial4}. 
 On the other hand we observe two situations. 
Firstly, if the first three elements are in the same block  then we have $\varphi_3((\beta\mathbb{X}-\alpha\mathbb{Y})(\X+\Y)(\beta\mathbb{X}-\alpha\mathbb{Y})(\mathbb{X+Y})^n)$. In the second case, if the first element and third one are is in the same block but not along with the second one then  we get $\varphi_2((\beta\mathbb{X}-\alpha\mathbb{Y})^2(\mathbb{X+Y})^n)$,  but taking into account that $\varphi(\X+\Y)= a $,we get $ a \varphi_2((\beta\mathbb{X}-\alpha\mathbb{Y})^2(\mathbb{X+Y})^n)$. By $\beta R_k(\X)= \alpha R_k(\Y)$ and $c$-free we also obtain (see the proof of Lemmas \ref{lem:3} and \ref{lemm:wlasnosciFunkcjiFi})
 \begin{align}
&\varphi((\beta\mathbb{X}-\alpha\mathbb{Y})(\X+\Y)(\beta\mathbb{X}-\alpha\mathbb{Y})(\mathbb{X+Y})^n)
\nonumber \\&= 
\varphi_3((\beta\mathbb{X}-\alpha\mathbb{Y})(\X+\Y)(\beta\mathbb{X}-\alpha\mathbb{Y})(\mathbb{X+Y})^n)+ a \varphi_2((\beta\mathbb{X}-\alpha\mathbb{Y})^2(\mathbb{X+Y})^n)
\nonumber \\&=
\alpha\beta\varphi_3((\mathbb{X+Y})^{n+3})+\alpha\beta a \varphi_2((\mathbb{X+Y})^{n+2}). \label{eq:PomocniczyNormalnyKumulanty}
 \end{align}
The equation \eqref{eq:TwierdzenieRozkladyRozdzPropo3} is equivalent to
$
C^{(2)}_{\varphi,\psi}(z)= b z^2M_\mu(z)  
$
where $C^{(2)}_{\varphi,\psi}(z)$ is a function for $\X+\Y$ (see Theorem \ref{twr:7}). So from \eqref{eq:exemK=2} we have $M_\nu(z)C^{(3)}_{\varphi,\psi}(z)=0$ (because $R_2(\mathbb{X+Y})= b $). But $M_\nu(z)\neq 0$ for sufficiently small $|z|$,  which gives $C^{(3)}_{\varphi,\psi}(z)=0$ thus we obtain \eqref{eq:TwierdzenieRozkladyRozdz3fipsiPropoPomoc} because  $\varphi_3((\mathbb{X+Y})^{n+3})=0$. If now \eqref{eq:TwierdzenieRozkladyRozdz3fipsiPropoPomoc} holds then 
$C^{(3)}_{\varphi,\psi}(z)=0$ and by \eqref{eq:exemK=2} we easily get \eqref{eq:TwierdzenieRozkladyRozdzPropo3}. 
\end{proof}

\section{Proof of the main theorem}
From Proposition \ref{twr:9} we see that it is sufficient to show that 1 $\Rightarrow$ 2 and 1 $\Leftarrow $ 2.
\begin{proof}

 1 $\Rightarrow$ 2: Suppose that $\X$ and $\Y$ƒ have  the two-state normal laws. 
Let's denote $\mu=\mu_1\boxplus\mu_2$ and $\nu=\nu_1\boxplus\nu_2$ then from Lemma \ref{lem:5} we get that the measure $\mu$ and $\nu$  
have Jacobi parameters \eqref{eq:JacobiDefiNoramal1} and \eqref{eq:JacobiDefiNoramal2}, respectively .
Using Proposition  \ref{lem:6}  we obtain  equation \eqref{eq:PierwszyOstatniorazVarphi} with  $C_{\shortparallel}(z)$  for $\X+\Y$. Expanding $M(z)$ and $C_{\shortparallel}(z)$ in the series and using the fact from the equation \eqref{eq:Pom2Propozycja1} (we skip a simple computation) we
get  
$$\varphi\big((\beta\X-\alpha\Y)(\X+\Y)^n(\beta\X-\alpha\Y))=\alpha\beta\varphi\big((1- b )\S^2+(-2 a + a  b )\S+( a ^2+ b ^2)\mathbb{I}\big).$$
\medskip

Now we prove \eqref{eq:PomocniczyNormalny2}. The moment generating function for $\X+\Y$ with respect to $\varphi$ satisfies (by the equation \eqref{eq:CiaglaFrakcja})
\begin{align}
zM_\mu(z)=\cfrac{1}{\frac{1}{z}- a -\cfrac{ b }{\frac{1}{z}-zM_{\sigma}(z)}}, \label{eq:RozwiniecieCiaglafrakcja}
\end{align}
where $M_{\sigma}(z)$ is the moment generating function for the measure $\sigma$ with Jacobi parameters 
\begin{align}
\left(\begin{array}{c c c c }
  0,& 0, & 0, &\dots  \\ 
 1,& 1, & 1, & \dots
\end{array}\right).\label{eq:Jacobi6} 
\end{align}
The equation \eqref{eq:RozwiniecieCiaglafrakcja} is equivalent to
\begin{align}
(M_\mu(z)(1-z a )-1)-(M_\mu(z)(1-z a )-1)z^2 M_{\sigma}(z)=z^2 b M_\mu(z) ,\label{eq:RozwiniecieCiaglafrakcja2}
\end{align}
From Lemma \ref{lem:2} we have  $M_\mu(z)-1-z a M_\mu(z)=M_\nu(z)C_{\varphi,\psi}^{(2)}(z)$, so 
\begin{align}
(M_\mu(z)(1-z a )-1)-M_\nu(z)C_{\varphi,\psi}^{(2)}(z)z^2 M_{\sigma}(z)=z^2 b M_\mu(z), \label{eq:RozwiniecieCiaglafrakcja2}
\end{align}
but $M_\nu(z)M_\sigma(z)z^2=M_\nu(z)-1$ (see the equation \eqref{eq:Rozdzial4pom1}), so
\begin{align}
(M_\mu(z)- a zM_\mu(z)-1)-(M_\nu(z)-1)C_{\varphi,\psi}^{(2)}(z)=z^2 b M_\mu(z) ,
\end{align}
then we  use again $M_\mu(z)-1-z a M_\mu(z)=M_\nu(z)C_{\varphi,\psi}^{(2)}(z)$ and we get
\begin{align}
C_{\varphi,\psi}^{(2)}(z)=z^2 b M_\mu(z) ,
\end{align}
which is equivalent to \eqref{eq:PomocniczyNormalny2}, because $\alpha\beta\varphi_2\big((\X+\Y)^{n+2}\big)=\varphi\big((\alpha\Y-\beta\X)^2(\X+\Y)^n\big)$. 
\\
 2 $\Rightarrow$ 1: Suppose now that the equalities \eqref{eq:PomocniczyNormalny2} and \eqref{eq:PomocniczyNormalny} hold. The relation \eqref{eq:PomocniczyNormalny}  is equivalent to \eqref{eq:PierwszyOstatniorazVarphi}. 
From Proposition \ref{lem:6} we deduce that the Jacobi parameters for $\mu$ (i.e. $\X+\Y$) is given by \eqref{eq:Jacobi3}, and 
$C_{\shortparallel}(z)$ satisfies the equation
$$(C_{\shortparallel}(z))^2- b C_{\shortparallel}(z)+ b ^2z^2=0.$$ 
\bigskip
From Theorem  \ref{twr:7} we obtain that \eqref{eq:PomocniczyNormalny2} is equivalent to $ 
C_{\shortparallel}(z)=M_\nu(z) b z^2$ so we get 
$$z^2M_\nu(z)-M_\nu(z)+1=0.$$ 
The  equation above is equivalent to 
%
\begin{align} 
M_\nu(z)=\frac{1-\sqrt{1-4z^2}}{2z^2}.
\end{align}
It is well known  that  the measure $\nu$ have Jacobi parameters  \eqref{eq:JacobiDefiNoramal2} (see \cite{BoBr,NS}).
From the assumption on  cumulants and $c$-free we see   $ r_{k}(\mathbb{X})=r_{k}(\X+\mathbb{Y})/\beta$ and $ r_{k}(\mathbb{Y})=r_{k}(\X+\mathbb{Y})/\alpha$ i.e.  cumulants disappear for $k>2$ (of course  we deduce similarly for $R_{k}(\X)$ and $R_{k}(\Y)$). Thus we see that   $\X$ and $\Y$ have  two-state normal distributions 
which proves the theorem. 
\end{proof}

\noindent \textbf{Open problems and remarks}
\begin{itemize}
\item In  this paper we assume that the measures $\mu$ and $\nu$ have compact supports. It would be interesting  to show if this measures can be replaced by any probability measure.
\item A version of Theorem \ref{twr:8} can be extended for two-state Meixner random variable (see \cite{AnMlotko}). The proof of this theorem is analogous to the proof of Theorem \ref{twr:8} (technical and demanding tools in this article). 
\item It would be interesting to show that   Theorem \ref{twr:8} above is true for two-state free Brownian motion.The existence of such  process, far from being trivial, is ensured by Anshelevich \cite{An3}.
\end{itemize}

\section*{Acknowledgment}
The author would like to thank Z. Michna, M. Bo\.zejko, W. Bryc,W. M\l{}otkowski and J. Wysocza\'nski  for several discussions and helpful comments
during preparation of this paper.



\begin{thebibliography}{0}
\bibitem{AcardiBozejko} Accardi L., Bo\.zejko M.: Interacting Fock spaces and Gaussianization of probability measures. Infin.
Dimens. Anal. Quantum Probab. Relat. Top., 1(4), 663-670, (1998).
\bibitem{AkhGlaz} Akhiezer N.I., Glazman I.M.: Theory of Linear Operators in Hilbert Space. Ungar, New
York, (1963).
\bibitem{An2} Anshelevich M.: Appell polynomials and their relatives. II. Boolean theory, Indiana Univ.Math. J., 58(2), 929-968, (2009).
\bibitem{An4} Anshelevich M.: Appell polynomials and their relatives. III. Conditionally free theory, Illinois J. Math., 53(1),
39-66, (2009).
\bibitem{An1} Anshelevich M.: Free evolution on algebras with two states. J. Reine Angew. Math., 638, 75-101, (2010). 
\bibitem{An3} Anshelevich M.: Two-state free Brownian motions. J. Funct. Anal., 260, 541-565, (2011).

\bibitem{AnMlotko} Anshelevich M.,  M\l{}otkowski W.: Semigroups of distributions with linear Jacobi parameters. J. Theoret. Probab., 25, 1173-1206, (2012).

\bibitem{Bozej1} Bo\.zejko M.: Positive definite functions on the free group and the noncommutative Riesz
product. Boll. Un. Mat. Ital. A, (6), 5(1):13-21, (1986).
\bibitem{Bozej2} Bo\.zejko M.: Uniformly bounded representations of free groups. J. Reine Angew. Math.
377, 170-186, (1987).
\bibitem{BoLR}  Bo\.zejko M., Leinert M., Speicher R.: Convolution and limit
theorems for conditionally free random variables. Pacific J. Math., 175(2), 357-388, (1996).
\bibitem{BoWys}  Bo\.zejko M., Wysocza\'nski J: Remarks on $t$-transformations of measures and convolutions. Ann. Inst. H. Poincaré Probab. Statist., 37 (6), 737-761, (2001). 
\bibitem{BoBr} Bo\.zejko M., Bryc W.: On a class of free L\'evy laws related to a regression problem. 
J. Funct. Anal., 236, 59-77, (2006). 
\bibitem{BoBr2} Bo\.zejko M., Bryc W.: A quadratic regression problem for two-state algebras with an application to
the central limit theorem. Infin. Dimens. Anal. Quantum Probab. Relat. Top., 12(2), 231-249, (2009).
\bibitem{Br} Bryc W.: Classical Versions of q-Gaussian Processes: Conditional Moments and Bell's Inequality. Comm.
Math. Phys.,  219(2), 259-270, (2001).
\bibitem{Chihara}  Chihara T.S.: An Introduction to Orthogonal Polynomials. Math. Appl., vol. 13, Gordon and Breach Science Publ.,
New York, (1978).
\bibitem{DelengaSniady} Do\l ega M., Féray V., \'Sniady P.: Explicit combinatorial interpretation of Kerov character polynomials as numbers of permutation factorizations. Adv. Math., 225(1), 81-120, (2010).
\bibitem{Ejs} Ejsmont W.: Laha-Lukacs properties of some free processes. Electron. Commun. Probab., 17 (13),1-8, (2012). 
\bibitem{Ejs2} Ejsmont W.: Characterizations of some free random variables by properties of conditional moments of third degree. To be published by the J. Theoret. Probab., (2012) DOI-10.1007/s10959-
012-0467-7.
\bibitem{Hasebe} Hasebe T.: Conditionally monotone independence I: Independence, additive convolutions and related convolutions. Infin. Dimens. Anal. Quantum Probab. Relat. Top., 14(3), 465-516, (2011).
\bibitem{HorObata}  Hora A., N. Obata: Quantum Probability and Spectral Analysis of Graphs. Springer, (2007).

\bibitem{JasiulisKula} Jasiulis-Go\l dyn B., Kula Anna.: The Urbanik generalized convolutions in the non-commutative probability and a forgotten method of constructing generalized convolution. Proc. Indian Acad. Sci. Math. Sci., 122(3), 437-458, (2012). 
\bibitem{KulaWysocz} Kula A., Wysocza\'nski J.: Noncommutative Brownian motions indexed by partially ordered sets. Infin. Dimens. Anal. Quantum Probab. Relat. Top. 13 (4), 629-661, (2010).
\bibitem{KrystekWoja1} Krystek A., Wojakowski L.: Associative convolutions arising from conditionally free convolution. Infin. Dimens. Anal. Quantum. Probab. Relat. Top., 8(3), 515-545, (2005). 
 \bibitem{KrystekWoja2} Krystek A., Wojakowski L.: Associative convolutions arising from conditionally free convolution. Infin. Dimens. Anal. Quantum. Probab. Relat. Top.,   8(4), 651-657, (2005).

\bibitem{Krystek1} Krystek A.: Infinite divisibility for the conditionally free convolution.  Infin. Dimens. Anal. Quantum. Probab. Relat. Top., 10(4),  499-522, (2007).
\bibitem{LL}  Laha R. G., Lukacs E.: On a problem connected with quadratic regression. Biometrika, 47(300), 335-343, (1960).
\bibitem{Lenczewski0} Lenczewski R., Salapata R.: Noncommutative Brownian motions associated with Kesten distributions and related Poisson processes. Infin. Dimens. Anal. Quantum Probab. Relat., Top. 11(3), 351-375, (2008).
\bibitem{Lenczewski} Lenczewski R.: Decompositions of the free additive convolution. J. Funct. Anal., 246(2), 330-365, (2007).
\bibitem{Lenczewski1}  Lenczewski R.: Matricially free random variables. J. Funct. Anal., 258(12), 4075-4121, (2010). 
\bibitem{Michna} Michna Z.: On the mean of a stochastic integral with non-Gaussian $\alpha$-stable noise.  Stoch. Anal. Appl. 27, (2), 258-269, (2009).
\bibitem{Mlotko2}  M\l{}otkowski W.: Operator-valued version of conditionally free product. Studia Math., (1),153, 13-30, (2002).
\bibitem{Mlotko1} M\l{}otkowski W.:Combinatorial relation between free cumulants and Jacobi parameters. Infin. Dimens.Anal. Quantum Probab. Relat. Top., 12(2), 291-306, (2009). 
\bibitem{NS}  Nica A.,  Speicher R.: Lectures on the Combinatorics of
Free Probability, London Mathematical Society Lecture Notes
Series 365, Cambridge University Press, (2006).
\bibitem{PopaWang}  Popa M.,  Wang J-Ch.: On multiplicative conditionally free convolution. Trans. Amer. Math. Soc. 363, 6309-6335, (2011).  
\bibitem{SY}  Saitoh N., Yoshida H.: The infinite divisibility and orthogonal
polynomials with a constant recursion formula in free probability theory.
Probab. Math. Statist., 21(1), 159-170, (2001).
\bibitem{SzWes}  Szpojankowski K.,  Weso\l owski J: Dual lukacs regressions for non-commutative variables. 
J. Funct. Anal., 266, 36-54, (2014). 
\bibitem{Wesolowski1} Weso\l{}owski, J.: Stochastic processes with linear conditional expectation and quadratic conditional variance. Probab. Math. Statist., 14(1), 33-44, (1993).
\bibitem{Wojakowski}  Wojakowski L.: Probability interpolating between free and Boolean. Dissertationes Math. (Rozprawy Mat.), 446, 45, (2007).
\end{thebibliography}
\end{document}